%% file: main.tex
\documentclass{gtpart}     
\usepackage{pinlabel}  
\usepackage{graphicx}  
\usepackage{subcaption}
\usepackage{amscd}

\usepackage{amsmath}
\usepackage{amssymb}
\usepackage{amsthm}
\usepackage{tabularx}
\usepackage{longtable}
\usepackage{booktabs}
\usepackage{tikz-cd}
\usepackage[normalem]{ulem}
\title{Unknotting $3$-Balls in the $5$--Ball}

\author{Daniel Hartman}
\surname{Hartman}
\address{Department of Mathematics\\ University
   of Georgia. }
\email{daniel.hartman@uga.edu}
\volumenumber{}
\issuenumber{}
\publicationyear{}
\papernumber{}
\startpage{}
\endpage{}
\doi{}
\MR{}
\Zbl{}
\received{}
\revised{}
\accepted{}
\published{}
\publishedonline{}
\proposed{}
\seconded{}
\corresponding{}
\editor{}
\version{}

\newtheorem{theorem}{Theorem}
\newtheorem{lemma}[theorem]{Lemma}

\newtheorem{corollary}[theorem]{Corollary}

\theoremstyle{definition}

\def\R{\mathbb R}

\newcommand{\Diff}{\mathop{\rm Diff}\nolimits}

\newcommand{\Emb}{\mathop{\rm Emb}\nolimits}

\begin{document}
\begin{abstract}
    The purpose of this note is to provide a positive answer to a question posed by Gay, as well as Hughes, Kim, and Miller regarding smooth embeddings of the $3$-ball in the $4$-sphere becoming isotopic relative to the bounding $2$-sphere when pushed into the $5$-ball.
\end{abstract}
\maketitle

\input{Sketch}

\bibliographystyle{plain}
\bibliography{3Ballin5Ball}
\end{document}

%% file: Sketch.tex


 In this note, we address a question raised by Gay and by Hughes, Kim, and Miller concerning the isotopy class of smoothly embedded 3-balls in the 4-sphere relative to the bounding 2-sphere when pushed into the 5-ball. We provide an affirmative answer by showing that any 3-ball in the 4-sphere is isotopic to a standardly embedded 3-ball in the 5-ball.

\begin{theorem}\label{T: Unknot B^3}
Let $~\mathcal{U}$ be the unknotted $S^2$ in $S^4$. Let $B$ and $B'$ two embedded $3$-balls with $\partial B=\partial B'=\mathcal{U}$. Then $B$ is isotopic to $B'$ rel $\mathcal{U}$ in $B^5$.
\end{theorem}

Using a standard innermost disk argument, one can show that the 1-unknot bounds a unique disk up to isotopy in $S^3$. However, in \cite{BG-3ball}, the authors showed that the same result is false for the 2-unknot in $S^4$. They demonstrated that for the unknotted 2-sphere, there are infinitely many non-isotopic 3-balls that bound the same 2-sphere. A consequence of Budney-Gabai's construction was that all of their examples became isotopic when pushed into the 5-ball. This raised the question of whether any 3-ball could remain knotted. In \cite{hughes2021knotted}, Hughes, Kim, and Miller answered this question for most higher genus surfaces. Specifically, they showed that for every unknotted surface of genus at least 2, there exists a pair of handlebodies, $H_1$ and $H_2$, which are homeomorphic as 3-manifolds, both smoothly embedded in $S^4$ with $\partial H_1=\partial H_2$, but that $H_1$ is not isotopic to $H_2$ relative to $\partial H$, even when pushed into the 5-ball. 



The unknotting of 3-balls arose from a different context for Gay. In section 4 of \cite{gay2021diffeomorphisms}, Gay introduces the notion of 5-dimensional dotted 1- and 2-handles Considering $S^4$ as the boundary of $B^5$, we can push the interior of any 3-ball bounding $\mathcal{U}$ into the 5-ball, and then carve out a neighborhood. The resulting 4-manifold boundary is $S^1\times S^3$, and the main theorem of this note implies that the 5-manifold is, in fact, $S^1\times B^4$. Just like carving a boundary parallel disk out of the 4-ball, carving out any $B^3$ bounding $\mathcal{U}$ is equivalent to attaching a 5-dimensional 1-handle.


In this paper, we work exclusively in the smooth category. Implicit in all the proofs below is the fact that corners can be smoothed. We will use the notation $\Diff(X,\partial X)$ to refer to the group of diffeomorphisms that fix the boundary of the given manifold $X$ pointwise. If we have an embedding $f:\partial N\rightarrow \partial X$, then we will denote the space of neat embeddings that coincide with $f$ by $\Emb_f(N,X)$. For the space $S^1\times B^3$, we fix an embedding $i:S^2\rightarrow \partial (S^1\times B^3)$ given by $i(s) = (z_0,s)$, where $z_0$ is a fixed point. This map is simply the restriction of the inclusion map $i:{z_0}\times B^3\rightarrow S^1\times B^3$. We will abuse notation by using the same symbol $i$ to denote various maps in different places throughout the paper. The main argument for unknotting any $B^3$ is derived from the fact that any nonseparating $S^3$ in $S^1\times S^3$ extends to a smoothly embedded $B^4$ in $S^1\times B^4$. This is a consequence of the classification of pseudoisotopies of $S^1\times S^3$ (\cite{lashof1969classification}). 

\begin{lemma}[Budney-Gabai \cite{BG-3ball}] The group $\Diff(S^1\times B^3, \partial)$ acts transitively on $\Emb_{i}(B^3, S^1\times B^3)$.
\begin{proof}
Here we give an independent proof which differs from the one given in \cite{BG-3ball}. Let $f\in \Emb_{i}(B^3, S^1\times B^3)$ and let $i:B^3\rightarrow S^1\times B^3$ be the inclusion of $\{z_0\}\times B^3$. Since $f$ is a neat embedding, we have that $f = i$ on some collar neighborhood, $C$, of the boundary. Now, using the triviality of the normal bundle of $f(B^3)$, we extend our embedding to a smooth embedding $F:[z_0 - \epsilon, z_0+\epsilon]\times B^3\rightarrow S^1\times B^3$ such that $F$ restricts to $f$ on $\{0\}\times B^3$ and $F$ agrees with the inclusion of $[z_0 - \epsilon, z_0+\epsilon]\times B^3$, in $C$. 

Now, we define an embedding, $\Phi: C\cup (z_0 - \epsilon, z_0+\epsilon)\times B^3 \rightarrow S^1\times B^3$ by 
\[
\Phi(x) = 
\begin{cases}
\operatorname{id}(x) & x\in C\\
F(x) & x \in (z_0 - \epsilon, z_0+\epsilon)\times B^3
\end{cases}
\]
In the open subset $C\cup (z_0 - \epsilon, z_0+\epsilon)\times B^3\subset S^1\times B^3$, there is a smoothly embedded $S^3$, given by smoothly rounding the corners of the piecewise smooth set, 
\[
\big\{z_0-\frac{\epsilon}{2}\big\}\times B^3 \cup \{z_0+\frac{\epsilon}{2}\}\times B^3 \cup \partial(S^1\times B^3)\setminus \left(z_0-\frac{\epsilon}{2},z_0+\frac{\epsilon}{2}\right)\times S^2,
\]
and pushing the resulting sphere into the interior (see Figure \ref{F: Smoothing}). Let $S$ denote this embedded $3$--sphere. Note that $S$ bounds a smooth $B^4$ in $S^1\times B^3$. We now restrict $\Phi$ to the complement of the interior of the $4$--ball bounded by $S$.
\begin{figure}[ht]
    \centering
    \includegraphics[width = .5\linewidth]{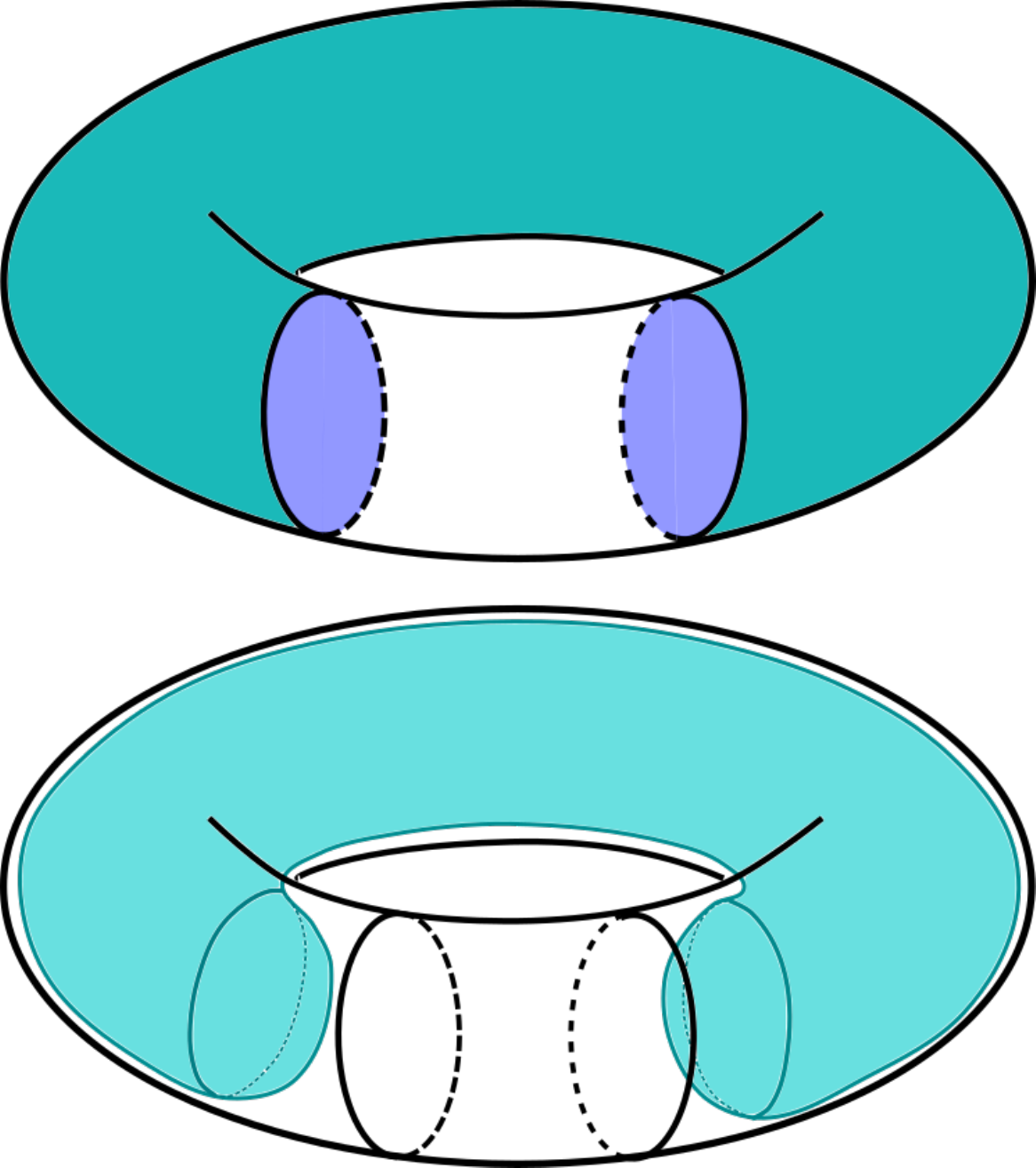}
    \caption{Smoothing of the $3$--sphere $S$.}
    \label{F: Smoothing}
\end{figure}

\textbf{Claim} If $\Phi(S)$ also bounds a smooth $4$--ball, then $\Phi$ extends to a diffeomorphism of all of $S^1\times B^3$ to itself. The claim follows from the fact that every diffeomorphism of $S^3$, extends to a diffeomorphism of $B^4$ \cite{cerf}. So, it only remains to show that $\Phi(S)$ bounds a smooth $4$--ball.

To show that $\Phi(S)$ bounds a smooth $4$--ball, we will show that, when considered as a subset of $S^4$, it bounds a smooth $4$--ball on one side, and hence on both sides. To do this, we attach to $S^1\times B^3$, a $4$--dimensional $2$--handle along the curve $S^1\times \{x_0\}$ with $x_0\in \partial B^3$, followed by a $4$--dimensional $4$--handle. Observe that, by definition, $\Phi = \operatorname{id}$ on $\partial (S^1\times B^3)$. Hence, we may extend $\Phi$ by the identity over the $2$-- and $4$--handles. Now since $S$ bounds a smooth $4$-ball, its complement in $S^4$ bounds a smooth $4$--ball. This is precisely the domain of $\Phi$ after we extend it over the $2$-- and $4$--handles. Thus, the extended $\Phi$ embeds $B^4$ into $S^4$. Since this is a smooth embedding, the complement has to be a smooth $4$-ball as well. This completes the proof.
\end{proof}
\end{lemma}



\begin{corollary}[Budney-Gabai \cite{BG-3ball}]\label{C: Trans-Prop}
The group $\Diff(S^1\times S^3)$ acts transitively on the nonseparating $S^3$'s.
\end{corollary}

What the corollary tells us is that, given any embedding of a nonseparating $S^3$ in $S^1
\times S^3$, we can realize this embedding by restricting a diffeomorphism of $S^1\times S^3$ to some $\{z_0\}\times S^3$. 
\begin{lemma}\label{L: Extension property}
Every nonseparating $S^3$ in $S^1\times S^3$ extends to a proper embedding of $B^4$ in $S^1\times B^4$.
\begin{proof}
Let $f:S^3 \rightarrow S^1\times S^3$ be an embedding of a nonseparating $S^3$ . By Corollary \ref{C: Trans-Prop}, there is a diffeomorphism, $F$, of $S^1\times S^3$ such that $F$, when restricted to $\{z_0\}\times S^3$, is the embedding $f$. Finally, by \cite{lashof1969classification}, $F$ extends to a diffeomorphism of $S^1
\times B^4$. Since $\{z_0\}\times S^3$ bounds the smooth 4-ball $\{z_0\}\times B^4$, $F\big|_{\{z_0\}\times B^4}$ is then the extension.

\end{proof}
\end{lemma}
We are now ready to prove the main theorem.
\begin{proof}[Proof of Theorem \ref{T: Unknot B^3}]\label{P: Unknot B^3}
Let $Y = S^4\setminus \nu(\mathcal{U})$. As we are considering what happens in $B^5$, take a collar neighborhood $S^4\times I$ of the boundary and extend $Y$ to $Y\times I$. As $\mathcal{U}$ is the unknot, $Y$ is diffeomorphic to $S^1\times B^3$ and $Y\times I$ is then diffeomorphic to $S^1\times B^4$, which in our given product structure is just $(S^1\times B^3)\times I$. Now consider the 3--spheres, $B\times\{0\}\cup \mathcal{U}\times I \cup B'\times \{1\}$. By Lemma \ref{L: Extension property}, the 3--sphere bounds a smoothly embedded 4--ball. Hence, we can isotope, rel $\mathcal{U}$, $B$ to $B'\cup \mathcal{U}\times I$. Then, using the product structure, isotope $B'\cup \mathcal{U}\times I$ to $B'$ (See Figure \ref{F:3Ball Iso}). 
\begin{figure}
    \centering
    \includegraphics[width=.5\linewidth]{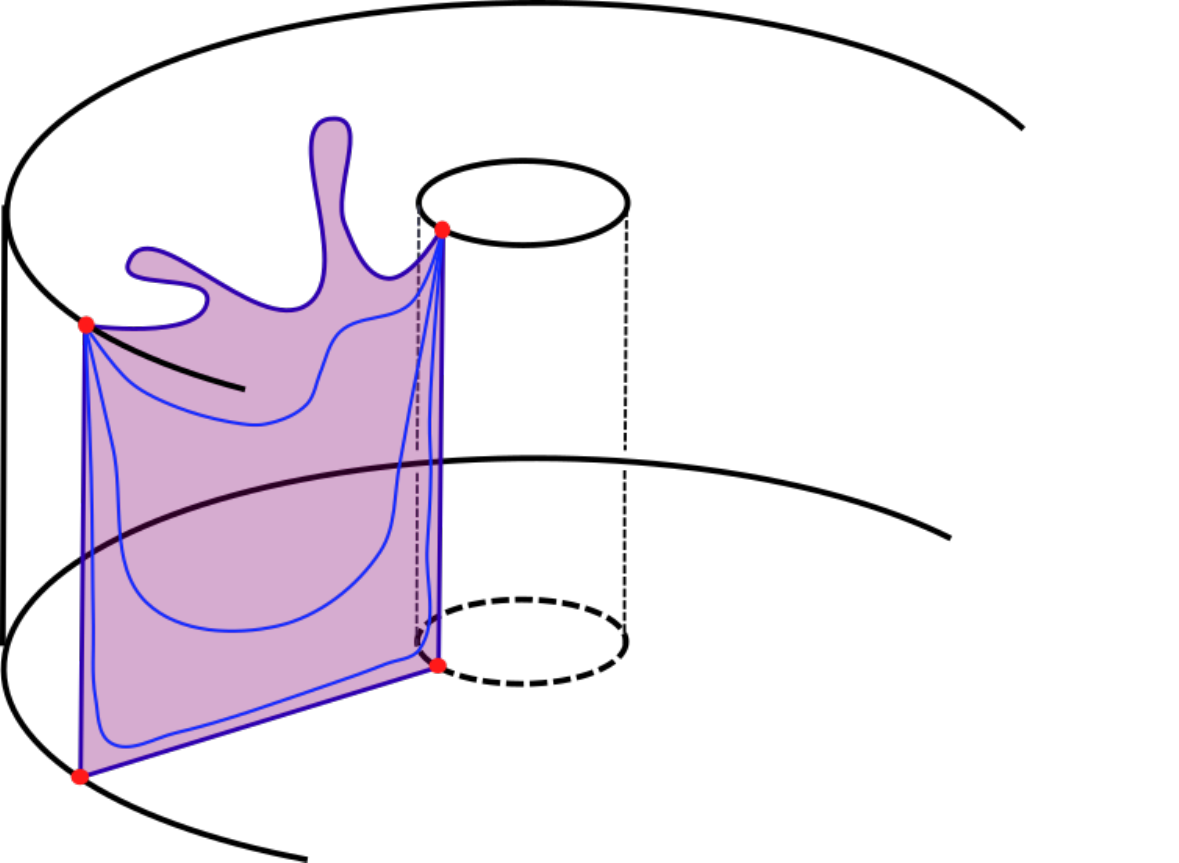}
    \caption{This figure is a cartoon depicting how the 4--ball bounding the $S^3$ inside of $S^1\times B^4$ is used to isotope between the two 3--balls.}
    \label{F:3Ball Iso}
\end{figure}
\end{proof}




The final corollary presented in this note can be viewed as a slice Schoenflies result. It is considered a classical result in the sense that experts in the field have known that it follows from other results. However, we were unable to find an explicit reference for this result in the literature, so we include it here to fill this gap.
 
\begin{corollary}\label{C: Slice_Schoenflies}
Every $S^3$ in $S^4$ bounds a smooth $B^4$ in $B^5$.

\begin{proof}
We begin by considering $S^4 = \partial B^5$. As every embedded $S^3$ separates $S^4$ into two connected pieces, we can assume that the $S^3$ separates the north and south poles of $S^4$. Attach an oriented 5--dimensional 1--handle to $B^5$ with the attaching sphere being the two poles and attaching regions disjoint from the $S^3$. The result is $S^1\times B^4$ and by Theorem \ref{L: Extension property}, the embedded $S^3$ bounds a smooth $B^4$. Denote this four ball by $\tilde{B}^4$. Even though $\tilde{B}^4$ is properly embedded, it might not be contained entirely in the original $B^5$ (which is given by compressing the two $B^4$'s used to attach the 1--handle). To fix this, we now consider the universal cover $\R\times B^4\rightarrow S^1\times B^4$, and the lift $\tilde{B}^4$ to the universal cover. As the embedded $S^3$ lies between two standard $S^3$'s, we can isotope $\tilde{B}^4$ to lie  between the two standard $B^4$'s. Projecting down gives a new $4$--ball which is contained in the original $B^5$.
\end{proof}
\end{corollary}